\documentclass[11pt]{amsart}
\usepackage[utf8]{inputenc} 
\usepackage[pctex32]{graphics}
\usepackage{amsfonts}
\usepackage{amssymb,amscd,latexsym}
\usepackage{amsmath}
\usepackage{epsfig}
\usepackage{color}

\newtheorem{thm}{Theorem}[section]
\newtheorem{cor}[thm]{Corollary}
\newtheorem{prop}[thm]{Proposition}

\newtheorem{defin}[thm]{Definition}

\newtheorem{lema}[thm]{Lemma}
\newtheorem{rmk}[thm]{Remark}

\newtheorem{ex}[thm]{Example}

\newcommand{\codim}{\operatorname{codim}}

\def\N{\mathbb{N}}

\def\C{\mathbb{C}}

\def\P{\mathbb{P}}
\def\K{\mathbb{K}}
\def\L{\mathbb{L}}
\def\rk{\operatorname{rk}}

\def\Hess{\operatorname{Hess}}
\def\hess{\operatorname{hess}}

\newcommand{\Hilb}{\operatorname{Hilb}}

\newcommand{\Ann}{\operatorname{Ann}}

\newcommand{\rank}{\operatorname{rank}}
\newcommand{\wrk}{\operatorname{wrk}}

\newcommand{\ba}{\mathcal{B}}

\begin{document}

\title{Waring problems and the Lefschetz properties}

\author[T. Dias]{Thiago Dias}

\author[R. Gondim]{Rodrigo Gondim}

\address{Universidade Federal Rural de Pernambuco, av. Don Manoel de Medeiros s/n, Dois Irmãos - Recife - PE
52171-900, Brasil}
\email{rodrigo.gondim@ufrpe.br}
\email{thiago.diasoliveira@ufrpe.br} 

\begin{abstract} We study three variations of the Waring problem for polynomials, concerning the Waring rank, the border rank and the cactus rank of a form and we show how the Lefschetz properties of the associated algebra affect them. The main tool is the theory of mixed Hessian matrix. We construct new families of wild forms, that is, forms whose cactus rank, of schematic nature, is bigger then the border rank, defined geometrically. 
\end{abstract}

\maketitle

\section*{Introduction}

The Waring problem, in number theory, asks for each exponent $k$, the minimum $s$ such that every positive integer can be decomposed as a sum of at least 
$s$ perfect $k$-th powers. In analogy, the algebraic Waring problem asks what is the minimum $s$ such that any homogeneous polynomial $f\in \K[x_0,\ldots,x_n]_d$, of degree $d$, can be decomposed as a sum of at least $s$ $d$-th powers of linear forms.  \\

The Waring problem for polynomials is a classical subject in Commutative Algebra and Algebraic Geometry and it has lots of variants. One of them is the following: for a given form $f$ of degree $d$, to find the minimal number $s$, such that $f$ can be decomposed as a sum of $s$ powers of linear forms. It goes back to Sylvester, that solves the problem for binary forms in  \cite{Syl,Syl2} (see also \cite{CS}). An explicit decomposition for a given polynomial is hard to find. For monomials there is a decomposition given in \cite{BBT,ECG}, but this decomposition sometimes is not be minimal one. The Waring problem was solved for generic forms by Alexander and Hirschowitz in \cite{AH1,AH2,AH3}. There are several applications of Waring problems in computational and applied Mathematics (see \cite{BCMT,CGLM}). \\

In our context, we are interested in three variants of the Waring problem. We work over the complex numbers.
Let $f \in R=\C[x_0, \ldots,x_n]$ be a degree $d$ form. We consider these notions of rank for $f$:
\begin{enumerate}
 \item[(i)] The Waring rank of $f$ is its algebraic rank: it is the minimum $s=wrk(f)$ such that $f$ can be decomposed as a sum of $d$-th powers of $s$ linear forms. 
 
 \item[(ii)] The Border rank of $f$ is its geometric rank: it is the minimum $s=\underline{rk}(f)$ such that the class of $f$ in $\P(R_d)$, where $R_d=\C[x_0,\ldots,x_n]_d$, belongs to the $s$-th secant variety of the Veronese image $\mathcal{V}_d(\P^n)\subset \P(R_d)$. It is equivalent to say that there is a one parameter family of forms $f_t$ of Waring rank $s$ such that $f=\displaystyle \lim_{t \to 0} f_t$. 
 
 \item[(iii)] The Cactus rank of $f$ is its schematic rank: it is the minimum $s=cr(f)$ such that there is a finite scheme $K$  of length $s$, $K \subset\mathcal{V}_d(\P^n))\subset \P(R_d)$ such that $[f] \in <K>$.  
\end{enumerate} 

It follows that $\underline{rk}(f)\leq wrk(f)$ and $cr(f)\leq wrk(f)$, while in general $cr(f)$ and $\underline{rk}(f)$ are incomparable (see \cite{BBM}). We are interested in special forms for which these notions of rank do not coincide. 
For instance, very few examples are known satisfying $cr(f)>\underline{rk}(f)$, they are called wild forms (see \cite{BB,HMV}). The main goal of this work is to describe new classes of wild forms, and to show how they are deeply connected with the Lefschetz properties of an associated algebra. To be precise, we are not using the usual definition of wild form, but our condition implies the usual one (see Definition \ref{def:wild} and the later comments and also \cite{BB,BBM,HMV}). \\

The Strong Lefschetz property (SLP) is an algebraic abstraction introduced by Stanley in \cite{St} for standard graded Artinian algebras. It was inspired by the  so called hard Lefschetz Theorem on the cohomology of smooth projective complex varieties (see \cite{La} and \cite[Chapter 7]{Ru}). 
Let $A = \displaystyle \bigoplus_{k=0}^d A_k$ be a graded Artinian $\K$-algebra. We say that $A$ has the Strong Lefschetz property (SLP for short) if there exists a linear form $l \in A_1$ such that every multiplication map $\mu_{l^j}:A_k \to A_{k+j}$ has maximal rank. A weaker formulation is called Weak Lefschetz property (WLP). We say that $A$ has the WLP if there is a linear form $l \in A_1$ such that all the multiplication maps $\mu_{l}:A_k \to A_{k+1}$ have maximal rank (see \cite{HMMNWW}).\\

Of particular interest are Artinian algebras satisfying Poincaré duality, which can be characterized as standard graded Artinian Gorenstein algebras, AG algebras for short (see \cite{MW}). The choice of algebras satisfying Poincaré duality is natural in the context of the original Lefschetz result and also in several new contexts where the Lefschetz properties have been introduced over the years, in categories having a cohomology algebra. From the geometric perspective, Lefschetz properties were studied for Projective Varieties (see \cite{La,Ru}), Solvmanifolds (see \cite{Ka}), Arithmetic Hyperbolic manifolds (see \cite{Be}), subvarieties of Shimura varieties (see \cite{HL}). In Combinatorics, Lefschetz properties were introduced in the context of Simplicial complexes by Stanley in \cite{St,St2} and used in \cite{BN,GZ,KN} just to cite some. In Representation Theory the Lefschetz properties were posed for co-invariant rings of Coxeter groups \cite{NW}. Lefschetz properties are also related with the Sperner property (see \cite{HMMNWW,St}).

Focusing our attention in AG algebras, by Macaulay Matlis duality one knows that they are a quotient of a polynomial ring (described as ring of differential operators) by the annihilator of a single form. The main tools to understand the SLP and the WLP are the Higher Hessian matrix, introduced in \cite{MW}, that controls the SLP and the mixed Hessian matrix, introduced in \cite{GZ2}, that generalize the previous notion and control both WLP and SLP.
Our first result is a factorization of the Mixed Hessian matrix of a form in a power sum decomposition of a form, see Proposition \ref{prop:decompositionofhessian}. We use this decomposition to give a criterion of maximality of its rank (see Proposition \ref{thm:consequencesofWDWt}) and WLP (see Corollary \ref{cor:easy}). As a Corollary we obtain an inequality between the border rank and the Waring rank of certain forms (see Corollary \ref{cor:wrk_br}). In \cite{IK}, the authors used power sum decomposition to study AG algebras and {\it vice versa}. This idea have been used many times.

We study the border rank of a class of bi-graded forms that are closely related to the classical works of Gordan-Noether and Perazzo on forms with vanishing Hessian, for a detailed account on the subject see \cite{Go}. In Proposition \ref{prop:border} we give an upper bound for the border rank of these forms. \\

The main results of this work are Theorem \ref{thm:main_A} and Theorem \ref{thm:main_C} and their Corollaries, that produce new classes of wild forms (see \ref{cor1} and \ref{cor2}). In \cite{BB,HMV}, the authors studied wild forms of minimal border rank with vanishing Hessian. In \cite{HMV} they proved that every form with vanishing Hessian and minimal border rank is wild. We construct classes of wild forms whose border rank is not minimal and also classes whose Hessian is non vanishing. 
Since we get an upper bound for the border rank of a class of forms related with forms with degenerated mixed Hessian, our strategy was to find a lower bound for the cactus rank in the same philosophy of \cite{BB,HMV}. As it has been noticed before in \cite{BB,HMV}, in degree one, a natural ingredient to find a lower bound for the cactus rank, is to show that it is bigger than the Hilbert function on this degree, of the associated AG algebra. Generalizing this idea we look for an element in the saturation, in degree $k$, of the ideal generated by the graded parts of degree $k$ of the Macaulay dual of $f$. 
To get a lower bound to the Cactus rank we impose that the form is $k$-concise, meaning that the Hilbert function is maximal up to degree $k$.

\section{Preliminaries}
\subsection{Artinian Gorenstein algebras and Lefschetz properties}

Let $\K$ be a field of $\operatorname{char}(\K)=0$ and let $A = \displaystyle \bigoplus_{i=0}^d A_i$ be an Artinian $\K$-algebra with $A_d \neq 0$, we say that $A$ is standard graded if $A_0 =\K$ and $A$ is generated in degree $1$ as algebra. 
The Hilbert function of $A$ can be described by the vector  $\operatorname{Hilb}(A) = (a_0, a_1, \ldots, a_d)$, where $a_i=\dim A_i$. We say that $\operatorname{Hilb}(A)$ is unimodal if it has no valleys, that is, there exists $k$ such that $1\leq a_1\leq \ldots \leq a_k \geq a_{k+1}\geq a_d$.

\begin{defin}\rm
A standard graded algebra $A$ is Gorenstein if and only if $a_d = 1$ and the restriction of the multiplication of the algebra in complementary degree, that is, $A_i \times A_{d-1} \to A_d \simeq \K$ is a perfect paring for $i =0,1,\ldots,d$ (see \cite{MW}). 
\end{defin}

Macaulay-Matlis duality produces standard graded Artinian Gorenstein algebras. Let us recall this construction.
Let $f \in R = \K[x_0,x_1,\ldots,x_n]_d$ be a form of degree $\deg(f)=d \geq 1$ and let 
$Q=\K[X_0,X_1,\ldots,X_n]$ be the ring of differential operators associated to $R$. 
We define the annihilator ideal 
$$\operatorname{Ann}(f) = \{\alpha \in Q|\ \alpha(f)=0\}\subset Q.$$
The homogeneous ideal $\operatorname{Ann}(f)$ of $Q$ is also called Macaulay dual of $f$. We define $$A=\frac{Q}{\operatorname{Ann}(f)}.$$
 $A$ is a standard graded Artinian Gorenstein $\K$-algebra such that $A_j=0$ for $j>d$ and such that $A_d \neq 0$ (see \cite[Section 1,2]{MW}). 
We assume, without loss of generality, that $(\operatorname{Ann}(f))_1=0$. 

The Theory of Inverse Systems gives us the converse. A proof of this result can be found in \cite[Theorem 2.1]{MW}.

\begin{thm}{\bf \ (Double annihilator Theorem of Macaulay)} \label{G=ANNF} \\
Let $R = \K[x_0,x_1,\ldots,x_N]$ and let $Q = \K[X_0,X_1,\ldots, X_N]$ be the ring of differential operators. 
Let $A= \displaystyle \bigoplus_{i=0}^dA_i = Q/I$ be an Artinian standard graded $\mathbb K$-algebra. Then
$A$ is Gorenstein if and only if there exists $f\in R_d$
such that $A\simeq Q/\operatorname{Ann}(f)$.
\end{thm}

\begin{defin} \rm With the previous notation, let $A= \displaystyle \bigoplus_{i=0}^dA_i = Q/I$ be an Artinian Gorenstein $\K$-algebra with $I = \Ann(f)$, $I_1=0$ and $A_d \neq 0$. In this case, the form is called concise. The socle degree of $A$ is $d$ which coincides with the degree of the form $f$. By abuse of notation, we say that the codimension of $A$ is the codimension of the ideal $I \subset Q$ which, in this case, coincides with its embedding dimension, that is, $\codim A = n+1$.
\end{defin}

We now recall the so called Lefschetz properties for a standard graded Artinian Gorenstein $\K$-algebra.

\begin{defin}\rm  Let $A = \displaystyle \bigoplus_{i=0}^dA_i$ be a standard graded Artinian Gorenstein $\K$-algebra.
\begin{enumerate}
 \item[(i)] We say that $A$ has the Strong Lefschetz property (SLP) if there is $L \in A_1$ such that the $\K$-linear multiplication maps $\bullet L^{d-2i}:A_i \to A_{d-i}$ are isomorphisms for $i=1,\ldots,\lfloor \frac{d}{2}\rfloor$. 
 \item[(ii)] We say that $A$ has the Weak Lefschetz property (WLP) if there is $L \in A_1$ such that the $\K$-linear multiplication maps $\bullet L:A_i \to A_{i+1}$ are of maximal rank for $i=0,\ldots,d$. 
\end{enumerate}
 
\end{defin}

Let $A=Q/\Ann(f)$ be a standard graded Artinian Gorenstein $K$-algebra of socle degree $d$, 
Let $k\le l\leq d$ be two integers and let $\ba_k=(\alpha_1,\ldots,\alpha_{m_k})$ be a $K$-linear basis of $A_k$ and 
$\ba_l=(\beta_1,\ldots,\beta_{m_l})$ be a $K$-linear basis of $A_l.$

\begin{defin}\rm
We call mixed Hessian of $f$ of mixed order $(k,l)$ with respect to the basis $\ba_k$ and $\ba_l$ the matrix: 
  $$\Hess_f^{(k,l)}:=[ \alpha_i\beta_j(f)]_{m_k \times m_l}.$$
Moreover, we define $\Hess_f^k=\Hess_f^{(k,k)}$ and $\hess_f^k = \det(\Hess_f^k)$ the Hessian matrix of $k$-th order and 
the Hessian of $k$-th order of $f$ respectively. Note that $\hess_f=\hess_f^1$.
  \end{defin}

 The next result is a generalization of \cite[Theorem 4]{Wa1} and \cite[Theorem 3.1]{MW}.
 It was proved in \cite[Corollary 2.5]{GZ2}.

\begin{thm}\cite{GZ2} {\bf (Hessian criteria for Strong and Weak Lefschetz elements)}\label{thm:hessiancriteria}

Let $A = Q/\operatorname{Ann}_Q(f)$ be a standard graded Artinian Gorenstein algebra of codimension $n+1$ and socle degree $d$ and let $L = a_0x_0+\ldots+a_rx_r\in A_1$. 
The map $\bullet L^{l-k}: A_k \to A_l$, for $k < l \leq \frac{d}{2}$, has maximal rank if and only if the (mixed) Hessian matrix $\Hess_f^{(d-l,k)}(a_0,\ldots,a_r)$ has 
maximal rank. 
In particular, we get the following:
\begin{enumerate}
 \item {\bf (Strong Lefschetz Hessian criterion, \cite{Wa1}, \cite{MW})} $L$ is a strong Lefschetz element of $A$ if and only if 
$\hess^k_f(a_0,\ldots, a_r)\neq 0$ for all $k=0,1,\ldots, [d/2]$.
\item {\bf (Weak Lefschetz Hessian criterion)} $L \in A_1$ is a weak Lefschetz element of $A$ if and only if either 
$d=2q+1$ is odd and $\hess^q_f(a_0,\ldots, a_r)\neq 0$ or $d=2q$ is even and $\Hess^{(q-1,q)}_f(a_0,\ldots, a_r)$ has maximal rank.
\end{enumerate}
\end{thm}

\subsection{Waring rank, border rank and Cactus rank}

Let $f \in R=\C[x_0,...,x_n]_{d}$ be a form. Any  expression of the form $f=l_{1}^d+...+l_{k}^d$, where $l_{1},...,l_{k}$ are linear forms on $R$, will be called a power sum decomposition of $f$. 
\begin{defin}\rm
The \emph{Waring rank} of $f$ over $R$ is the least number of terms in a power sum decomposition of $f$, we denote it by $\operatorname{wrk}(f)$. 
\end{defin}

In \cite{Syl,Syl2} Sylvester determined the Waring rank of homogeneous polynomials of two variables, this results can be summarized in the following Theorem.

\begin{thm}{\bf (Sylvester)}\label{thm:sylvester} The Waring rank of a generic polynomial $f\in\mathbb{K}[x,y]_{d}$ is $\left \lceil{\frac{d-1}{2}}\right\rceil$.
\end{thm}

In \cite{AH1},\cite{AH2} and \cite{AH2}, Alexander and Hirschowitz described the Waring rank for a generic form.
\begin{thm}{\bf(Alexander-Hirschowitz)} A generic $f\in \C[x_0,\ldots,x_n]_d$ has Waring rank $\operatorname{wrk}(f) = \left \lceil{\frac{\binom{n+d}{d}}{n+1}}\right\rceil$, except for:
\begin{enumerate}
 \item[(i)] $(n,2)$, in this case $\operatorname{wrk}(f) =n+1$;
 \item[(ii)] $(n,d) = (2,4), (3,4), (4,3), (4,4)$, in this case $\operatorname{wrk}(f) = \left \lceil{\frac{\binom{n+d}{d}}{n+1}}\right\rceil+1$.
\end{enumerate}

\end{thm}

From a more geometric viewpoint we consider the following picture. Given a power sum decomposition $f=l_{1}^d+...+l_{k}^d$, 
consider $P_i = l_i^{\perp} \in \P^n$. We will identify the ideal of points of $\Gamma = \{P_1, \ldots, P_s\}$, $I_{\Gamma}$ with an ideal in $Q$, by the differential version of Macaulay-Matlis duality. Under this identification we have the following useful Lemma whose proof can be found in \cite[Lemma 1.31]{IK}.

\begin{lema}{\bf Apolarity Lemma} \label{lemma:apolarity}A form $f \in R_d$ ca be decomposed as\[f=l_{1}^d+...+l_{k}^d\] 
with $l_i$ pairwise linearly independent linear forms if and only if $I_{\Gamma } \subset \Ann_f$. 
\end{lema}

\begin{defin} \rm Let $X \subset \P^n$ be a projective variety. The $s$-th secant variety of $X$ is 
\[S^s(X) = \overline{\{<p_1,\ldots,p_s>|p_i \in X\}}\subset \P^n.\]
 
\end{defin}

Consider $R=\C[x_0, \ldots, x_n]$ and $R_d$ its graded part of degree $d$. 

\begin{defin}\rm The Veronese map $\mathcal{V}_d:\P(R_1) \to \P(R_d) $ is the morphism given by $\mathcal{V}_d([l])=[l^d]$. 
Its image is called the Veronese variety $\mathcal{V}_d(\P^n)$.
 
\end{defin}

\begin{defin}\rm Let $f \in R_d$ and $p =[f]\in  \P(R_d) $ the corresponding point. The border rank of $f$ is the minimal integer $s=\underline{rk}(f)$ such that $p \in S^s(\mathcal{V}_d(\P(R_d))$. 
 
\end{defin}

Notice that $\underline{rk}(f)=s$ means that $[f]$ is a limit of forms with Waring rank $s$.

In the sequel we will need the following result about the border rank of monomials. 

\begin{thm}\cite[Theorem 11.2]{LT} \label{thm:border_monomials} If $e_0\geq e_1 \geq \ldots \geq e_n$, then 
\[ \underline{rk}(x_0^{e_0}x_1^{e_1}\ldots x_n^{e_n})\leq (e_1+1)\ldots(e_n+1).\]
\end{thm}

\begin{defin}\rm Let $f \in R_d$ and $p =[f]\in  \P(R_d) $ the corresponding point. The cactus rank of $f$ is the minimal integer $s=cr(f)$ such that there is a length $s$ finite scheme $K \subset \mathcal{V}_d(\P(R_d))$ such that $p \in <K>$. 
 
\end{defin}

\begin{defin}\rm Let $f \in R_d$ and $p =[f]\in  \P(R_d) $ the corresponding point. The smoothable rank of $f$ is the minimal integer $s=sr(f)$ such that there is a a length $s$ {\it smoothable} finite scheme $K \subset \mathcal{V}_d(\P(R_d))$ such that $p \in <K>$. 
 
\end{defin}

\begin{rmk}\rm It is clear, from the definitions that $wrk(f) \geq \underline{r}(f)$ and that $cr(f)\leq sr(f)$. 
See \cite{BBM} for a detailed discussion about the relations among various notions of rank of a form. We know that:
\[\underline{rk}(f) \leq sr(f) \leq wrk(f).\]
\[cr(f)\leq sr(f) \leq wrk(f).\]
 Moreover, $cr(f)$ and $\underline{rk}(f)$ are incomparable. For instance, in \cite{BR} there are examples of forms for which 
 $cr(f) < \underline{rk}(f)$. On the other hand, in \cite{BB,HMV} and in the present work we give examples of forms for which 
 $ \underline{rk}(f)< cr(f)$, these forms are wild, in a sense that we precise in the third section. 
\end{rmk}

\begin{ex}\rm In \cite{BB} the authors showed that the form $f=xu^2+y(u+v)^2+zv^2 \in \C[x,y,z,u,v]$ has 
\[wrk(f)=9,\ \underline{rk}(f)=5\ \text{and} \ sr(f)=cr(f)=6.\]
Moreover, in \cite{HMV}, the authors showed that inequality $ \underline{rk}(f)< cr(f)$ was a consequence of two properties of $f$.
\begin{enumerate}
 \item[(i)] $f$ has minimal border rank, that is $\underline{rk}=a_1=5$;
 
 \item[(ii)] $\hess_f =0$.
\end{enumerate}
Concise cubic forms with vanishing Hessian were studied by Perazzo in \cite{Pe} and revisited in \cite{GRu}. 
In $\C[x,y,z,u,v]$ there is only one concise cubic form with vanishing Hessian up to projective transformations.  
 
\end{ex}

\section{Hessian matrices of a form in a power sum decomposition}

Let $R=\C[x_0,\ldots,x_n]$ be a polynomial ring and $Q=\C[X_0,\ldots,X_n]$ be the associated ring of differential operators. 
Let $f \in R_d$ be a form and let $A(f)=Q/Ann(f)$ be the associated AG algebra. Consider a power sum decomposition of $f$. 
$$f=l_1^d+l_2^d+\ldots+l_s^d.$$
We are considering $s\geq \wrk(f)$, that is, it is not necessarily the Waring decomposition.  \\ 

Let $\{\alpha_1,\ldots,\alpha_{m_k}\}$ be a basis of the $\C$-vector space $A_k$, and $\{\beta_1,\ldots,\beta_{m_l}\}$ be a basis of the $\C$-vector space $A_{d-l}$  for some $k<l\leq d-k$ and $k \in \{1,\ldots,\lfloor\frac{d}{2}\rfloor\}$. We can suppose without loss of generality that $m_k\leq m_l$, by unimodality.\\

For any linear form $l_r = \displaystyle \sum_{t=1}^n a_{tr}x_t$ and for any $\alpha_j=\displaystyle \prod_{t=1}^nX_t^{e_{tj}}$ we get: 
\begin{equation}\label{eq:derivadadapotencia}
 \alpha_j(l_r^d) = \frac{d!}{(d-k)!}l_r^{d-k}\displaystyle \prod_{t=1}^na_{tr}^{e_{tj}}.
\end{equation}

We define $w_{jr}^{(k)}=\displaystyle \prod_{t=1}^na_{tr}^{e_{tj}} \in \C$ for $j =1, \ldots, m_k$. For any $\beta_i=\displaystyle \prod_{k=1}^nX_t^{f_{ti}}$ let $w_{ir}^{(d-l)}=\displaystyle \prod_{k=1}^na_{tr}^{f_{ti}}$ with $i=1,\ldots,m_l$. Using Equation \ref{eq:derivadadapotencia}, we get:
 
\begin{equation}\label{eq:elementogeraldohessiano}
 \beta_i \alpha_j(l_r^d) = \frac{d!}{(l-k)!}l_r^{l-k}w_{ir}^{(d-l)}w_{jr}^{(k)}.
\end{equation} 

Let $ W_k= [w_{jr}^{(k)}]_{m_k \times s}$, $ W_{d-l}= [w_{ir}^{(d-l)}]_{m_l \times s}$ and  $D_{k,l}= \operatorname{Diag}(l_1^{l-k}, l_2^{l-k}, \ldots, l_s^{l-k})$. 
Sometimes we omit the index $(k)$ if it is clear in the context, especially when $l=d-k$.
\begin{lema} With the previous notations, we get:
\begin{enumerate}
 \item $$\Hess_f^{(d-l,k)} = \frac{d!}{(l-k)!} [W_{d-l}]_{m_l\times s}[D_{k,l}]_{s \times s}[W_k]_{s \times m_k}^t.$$
 
 \item $$\Hess_f^k = \frac{d!}{(d-2k)!} [W_{k}]_{m_k\times s}[D_{k,d-k}]_{s \times s}[W_k]_{s \times m_k}^t.$$
 
\end{enumerate}

\end{lema}

\begin{proof}\label{prop:decompositionofhessian}
 By definition $\Hess_f^{(d-l,k)}=(\beta_i \alpha_j(f))_{m_l \times m_k}$. Hence, 
 $$\Hess_f^{(d-l,k)}=(\beta_i \alpha_j(f))_{m_l \times m_k}=\frac{d!}{(l-k)!}(\displaystyle \sum_{r=1}^s l_r^{l-k}w_{ir}^{(d-l)}w_{jr}^{(k)})_{m_l \times m_k}=$$
 $$=\frac{d!}{(l-k)!}\left[\begin{array}{ccc}
     \displaystyle \sum_{r=1}^s l_r^{l-k} w_{1r}^{(d-l)}w_{1r}^{(k)} & \ldots & \displaystyle \sum_{r=1}^s l_r^{l-k} w_{1r}^{(d-l)}w_{mr}^{(k)}\\
     \ldots & \ldots & \ldots \\
     \displaystyle \sum_{r=1}^s l_r^{l-k} w_{mr}^{(d-l)}w_{1r}^{(k)} & \ldots & \displaystyle \sum_{r=1}^s l_r^{l-k} w_{mr}^{(d-l)}w_{mr}^{(k)}
    \end{array}\right]_{m_l \times m_k}
$$
$$=\frac{d!}{(l-k)!}\left[\begin{array}{ccc}
      l_1^{l-k} w_{11}^{(d-l)} & \ldots &  l_s^{l-k} w_{1s}^{(d-l)}\\
     \ldots & \ldots & \ldots \\
      l_1^{l-k} w_{m_{l}1}^{(d-l)} & \ldots &  l_s^{l-k} w_{m_{l}s}^{(d-l)}
    \end{array}\right]_{m_l \times s} \left[\begin{array}{ccc}
     w_{11}^{(k)} & \ldots & w_{m_{k}1}^{(k)}\\
     \ldots & \ldots & \ldots \\
     w_{1s}^{(k)} & \ldots &  w_{m_{k}s}^{(k)}
    \end{array}\right]_{s \times m_k} 
    $$
    
   $$=\frac{d!}{(l-k)!}\left[\begin{array}{ccc}
       w_{11}^{(d-l)} & \ldots &   w_{1s}^{(d-l)}\\
     \ldots & \ldots & \ldots \\
      w_{m_{l}1}^{(d-l)} & \ldots &  w_{m_{l}s}^{(d-l)}
    \end{array}\right]_{m \times s} \left[\begin{array}{ccc}
      l_1^{l-k}  & \ldots &  0\\
     \ldots & \ldots & \ldots \\
      0 & \ldots &  l_s^{l-k} 
    \end{array}\right]_{s \times s} \left[\begin{array}{ccc}
     w_{11}^{(k)} & \ldots & w_{m_{k}1}^{(k)}\\
     \ldots & \ldots & \ldots \\
     w_{1s}^{(k)} & \ldots & w_{m_{k}s}^{(k)}
    \end{array}\right]_{s \times m} $$ 
\end{proof}

\begin{rmk}\rm
 Sylvester proved in \cite{Syl} that $\wrk(f) \geq \rk(\Hess^k_f)$ for $k = \lfloor \frac{d}{2}\rfloor$ (see also \cite[Corollary 3.5]{Do}). If $A=A(f)$ has the SLP, it implies that $s \geq m_k$ for all $k$.

\end{rmk}

 Consider the natural exact sequence 
 \[0 \to I_k \to Q_k \to A_k \to 0.\]
 
 We can think $Q$ as a polynomial ring, in this context we identify $\P^n=\P(Q_1)$, $\P^{\nu(k,n)-1}=\P(Q_k)$  and $\P^{a_k-1}=\P(A_k)$. consider the Veronese map $\mathcal{V}_{k}:\P^{n} \to \P^{\binom{n+k}{k}}$ given by $\mathcal{V}_{k}(L)=L^k$. We get the following diagram:
 
 \[\begin{array}{ccc}
\P^n & \hookrightarrow & \P^{\nu(k,n)-1}\\
& & \downarrow \\ 
& & \P^{a_k-1}  
 \end{array}\]

 We consider the map $\mathcal{V'}_{k}:\P^n \to  \P^{a_k-1} $ the relative Veronese (see \cite{DGI}). 
 
\begin{prop}\label{P1}
Let $f=l_1^d+...+l_s^d \in R=\mathbb{K}[x_0,...,x_n]$ with $d>2k$, $k>1$ and $a_k=\dim A_{k}$. Let $P_i = l_i^{\perp}$ be the point that is dual of the hyperplane defined by $l_i$.
Consider the Veronese map $\mathcal{V}_{k}:\P^{n} \to \P^{\binom{n+k}{k}}$. Then, $$W_k=[[\mathcal{V}_{k}(P_1)]:...:[\mathcal{V}_{k}(P_s)]]_{a_k \times s}.$$ 
Moreover, $W_k$ has maximal rank. 
\end{prop}

\begin{proof}
 Note that $A_k$ has a monomial basis, let $\alpha=X_0^{c_0}x_1^{C_1}...X_N^{c_{N}}\in A_k$, $c_0+...+c_n=k$, and denote $l_{r}=(a_{1r}x_1+...+a_{nr}x_n)$, we have:

$$\alpha_{i_{1}...i_{k}}(l_{r}^{d})=\frac{d!}{(d-k)!}l_{r}^{d-k}a_{i_1}^{c_{i_{1}}}...a_{i_{k}}^{c_{ik}}.$$

Hence all the entries of the $rth$ column of $W$ are of the form
$w_{i_{1}...i_{r}}=a_{i_1}^{c_{i_{1}}}...a_{i_{k}}^{c_{i_k}}$. \\

The maximality of the rank follows from the Apolarity Lemma \ref{lemma:apolarity}. In fact, if the rank of $W_k$ drops, then, the image of $\Gamma$ by the relative Veronese, $\Gamma' \subset \P^{a_k-1}$, should satisfy $<\Gamma'> \subset H \subset \P^{a_k-1}$. It means that there is degree $k$ form $\alpha \in A_k $ in the ideal of the points $P_i$, but $I_{\Gamma} \subset \Ann_f=I$. The result follows. 

\end{proof}

\begin{prop}\label{thm:consequencesofWDWt}
 Consider the decomposition of the Hessian matrix: 
  $$\Hess_f^{(d-l,k)} = \frac{d!}{(l-k)!} [W_{d-l}]_{m_l\times s}[D_{k,l}]_{s \times s}[W_k]_{s \times m_k}^t.$$
  Assuming that $s\geq m_l \geq m_k$ we get:
 \begin{enumerate}
 \item If $s=m_k=m_l$, then $\text{det}(\Hess_f^{(d-l,k)}) \neq 0$.
     \item If $s>m_l$, then
  $$\operatorname{dim}\left(\operatorname{Im}(DW_k^t)\cap\operatorname{Ker}(W_{d-l})\right)=\operatorname{dim}\left(DW_k^t\left(\operatorname{Ker}\left(\Hess^{(d-l,k)}_f\right)\right)\right)=\operatorname{dim}\left(\operatorname{Ker}\left(\Hess^{(d-l,k)}_f\right)\right).$$
 Moreover, the following conditions are equivalent: 
\begin{enumerate}
 \item $\rank(\Hess_f^{(d-l,k)})$ is maximal;
 \item $ \operatorname{dim}\left(\operatorname{Im}(DW_k^t)\cap \operatorname{Ker}(W_{d-l})\right)= m_l-m_k$;
 \item $\operatorname{dim}\left(\operatorname{Im}(W_k^t)\cap \operatorname{Ker}(W_{d-l}D)\right)= m_l-m_k $.
\end{enumerate}

  \end{enumerate}

\end{prop}
\begin{proof}
Consider the decomposition of $\Hess_f^{(d-l,k)} = \frac{d!}{(l-k)!} [W_{d-l}]_{m_l\times s}[D_{k,l}]_{s \times s}[W_k]_{s \times m_k}^t.$ in a diagram of $\L$ vector spaces, with $\L=\C(x)$. Recall that $D$ is an isomorphism.

$$\begin{array}{ccccc}
 & & \Hess_f^{(d-l,k)} & & \\
 &  \L^{m_k} & \longrightarrow & \L^{m_l} & \\
W_{k}^t &  \downarrow & & \uparrow & W_{d-l} \\
 &   \L^s & \longrightarrow & \L^s & \\
 & & D & & 
   
  \end{array}
$$
\begin{enumerate}
\item If $s=m_k=m_l$, $\Hess_f^{(d-l,k)}$ is a square matrix. Since $\det(D)\neq0$, the result follows immediately from the decomposition formula.

 \item It is easy to check that $\operatorname{Im}(DW_k^t)\cap\operatorname{Ker}(W_{d-l})=DW_k^t\left(\operatorname{Ker}\left(\Hess^{(d-l,k)}_f\right)\right)$. Since $W_l^t$ and $D$ are injective, we have
 $$\operatorname{dim}\left(\operatorname{Im}(DW_k^t)\cap\operatorname{Ker}(W_{d-l})\right)=\operatorname{dim}\left(DW_k^t\left(\operatorname{Ker}\left(\Hess^{(d-l,k)}_f\right)\right)\right)=\operatorname{dim}\left(\operatorname{Ker}\left(\Hess^{(d-l,k)}_f\right)\right)$$
 $\Hess^{(d-l,k)}_f$ has maximal rank if and only if $\operatorname{dim}(\operatorname{Ker}\left(\Hess^{(d-l,k)}_f\right)=m_{l}-m_{k}$. Now we get $(a)\Leftrightarrow (b) \Leftrightarrow (c)$ .
\end{enumerate}
 \end{proof}

 \begin{cor} \label{cor:easy} Let $f \in R_d$ be a form and let $A$ be the associated AG algebra. Suppose that $f$ has a power sum decomposition  
  with $s=a_k$ for some $k \leq d/2$. Then \[\hess^k_f \neq 0.\]
  In particular, if $d=2q+1$ and $k=q$, then $A$ has the WLP. 
 \end{cor}

 \begin{proof}
  If follows from Proposition \ref{thm:consequencesofWDWt} and from the Hessian criteria Theorem \ref{thm:hessiancriteria}.
 \end{proof}

\begin{cor}\label{cor:wrk_br}
 Let $f \in R_d $ be a concise homogeneous form and $A=A(f)$ be the associated algebra. If $\underline{rk}(f)=\dim A_k$ and $\hess_f^k = 0$, then 
 \[wrk(f)>\underline{rk}(f).\]
 In particular, all concise forms of minimal border rank and vanishing Hessian have rank greater then its border rank. 
\end{cor}

\begin{proof}
 We get that $wrk(f)\geq \underline{rk}(f)$. If equality holds true, let $r=wrk(f)$. We get a limit $\displaystyle \lim_{t\to 0} l_i(t)=l_i\in A_1$ satisfying 
 \[f= \displaystyle \lim_{t\to 0} \sum_{i=1}^r l_i(t)^d= \displaystyle \sum_{i=1}^r \lim_{t\to 0} l_i(t)^d=\displaystyle \sum_{i=1}^r l_i^d\]
 On the other hand, if $wrk(f)=r$, then, by Corollary \ref{cor:easy}, $\hess_f^k \neq 0$. 
\end{proof}

\section{Wild forms}

The definition of a wild form can be found in \cite{BB,HMV}. 

\begin{defin}\label{def:wild}
 We say that a form $f \in R_d$ is wild if \[\underline{rk}(f) < sr(f) .\] 
\end{defin}

Since  $sr(f) \geq cr(f)$ and since we are not interested in the smoothable rank we produce wild forms showing that $\underline{rk}(f) < cr(f)$. Our strategy is to find an upper bound to $\underline{rk}(f)$ which is also a lower bound to $cr(f)$. That is, a positive integer $a$ such that
\[\underline{rk}(f) \leq a < cr(f).\]

The next result is a generalization of \cite[Proposition 2.6]{BB}.

\begin{prop}\label{prop:border}
 Let $X\subset \P^N$ be a projective variety of dimension $\dim X=n$ and let $x_1,\ldots,x_r \in X$ smooth points. Suppose that $\dim <x_1,\ldots,x_r> \leq r-1$. Then
 \[<T^k_{x_1}X, \ldots, T^k_{x_r}X> \subset S^rT^{k-1}X \subset S^{kr}X.\]
\end{prop}
\begin{proof} Since $T^{k-1}X \subset S^k X$, we have $S^r(T^{k-1}X) \subset S^r(S^k X) \subset S^{rk}(X)$. 
Take points $a_1$,...$a_r \in \mathbb{C}^{N+1}$ such that $[a_i]=x_i$, for $i=1,...,r$. Since $\dim <x_1,\ldots,x_r> \leq r-1$, we can suppose that $a_1+\ldots+a_r=0$. Let $v$ an arbitrary point of $<T^{k}_{x_1}X, \ldots, T^{k}_{x_r}X>$. We can write $v=v_{1}+...+v_{r}$ with $v_{i}\in T^{k}_{x_i}X$. Let $\alpha_i(t) \subset T^{k-1}X$ be a curve such that $\alpha_i(0)=x_i$ and $\alpha'_i(o)=v_i$. It is possible since the vectors of $T^{k}_{x_i}X$ belongs to the tangent cone of $T^{k-1}X$ in $x_i$. \\
Define the curve $\alpha(t)=\frac{1}{t}\sum_{i=1}^{r}\alpha_i(t)$. Note that $[\alpha(t)]\in S^r(T^{k-1}(X))$. Therefore:
\begin{align*}
S^r(T^k(X))\ni [\alpha(0)]&=\left[\displaystyle \lim_{t\rightarrow 0}\frac{1}{t}\sum_{i=1}^{r}(\alpha_i(t)-a_i)\right]\\
 =& \left[\sum_{i=1}^{r}\displaystyle \lim_{t\rightarrow 0}\frac{\alpha_i(t)-\alpha_i(0)}{t}\right]=\left[\sum_{i=1}^rv_{i}\right]=[v].\\
\end{align*}
\end{proof}

The next result is a generalization of \cite[Lemma 5.1]{HMV}.

\begin{cor}
 Let $f \in \C[x_1, \ldots,x_n, u, \ldots, v]_{(k,d-k)}$ be a bi-homogeneous form of bi-degree $(k,d-k)$ with $1\leq k \leq d-k$. 
 The border rank of $f$ satisfies: \[\underline{rk}(f)\leq k(d+2).\]
\end{cor}

\begin{proof}
 Since $\dim \C[u,v]_d = d+1$, let $l_0^d, \ldots,l_d^d \in \C[u,v]_d$ be a basis. It is easy to see that $f = \displaystyle \sum_{i=0}^d f_i(\underline{x})l_i^d$. Let $l_{d+1} \in \C[u,v]$ be an arbitrary linear form. The points $x_0=[l_0^d],\ldots, x_d=[l_d^d], x_{d+1}=[l_{d+1}^d] \in \mathcal{V}(d,\P^1)=X$ are linearly dependent, that is, $\dim <x_0,\ldots,x_{d+1}> \leq d+1$. Therefore, by Proposition \ref{prop:border}, 
 $<T^k_{x_0}X, \ldots, T^k_{x_{d+1}}X>  \subset S^{k(d+2)}X$.
 Since $[f] \in <T^k_{x_0}X, \ldots, T^k_{x_{d+1}}X> \subset S^{k(d+2)}X$, the result follows.
\end{proof}

\subsection{$k$-concise wild forms with vanishing Hessian}

\begin{defin}\rm 
 A form $f \in R_d$ is called $k$-concise, with $d \geq 2k+1$, if  $I_j=0$ for $j=1,2,\ldots,k$. It is equivalent to $a_j=\binom{n+j}{j}$ for $j=0,\ldots,k$. As usual, $1$-concise forms are called concise. 
\end{defin}

The following Lemma is a generalization, for higher Hessians of an idea contained in proof of \cite{HMV}[Theorem 3.5] for the case of classical Hessians.

\begin{lema}\label{lema:hess}
 Let $f \in R_d$ be a concise form and $A=A(f)= Q/I$ be the associated algebra. Suppose that $a_k\leq a_{d-s}$ and $k+s\leq d$. If $ \Hess_f^{(k,s)}$ is degenerated, then 
 exists $\alpha \in I_k^{sat}\setminus I_k$. 
\end{lema}

\begin{proof}
 We are considering $ \Hess_f^{(k,s)}$ as a matrix in $R$. By the Hessian criteria \ref{thm:hessiancriteria}, for each $L \in A_1$, the map 
 $\bullet L^{d-s-k}:A_k \to A_{d-s} $ is represented by $\Hess_f^{(k,s)}(L^{\perp})$. Therefore, there is a universal polynomial in the kernel of $\Hess_f^{(k,s)}$ such that its image $\alpha \in A_k $ belongs the kernel of $\bullet L^{d-s-k}$ for every $L \in A_1$, that is $L^{d-s-k}\alpha \in I_{d-s}$. In particular, $X_i^{d-k-s}\alpha \in I_{d-s}$ for $i=0,\ldots,n$, that is, $\alpha \in I^{sat}_k\setminus I_k$.   
\end{proof}

\begin{lema}\label{lema:k_concise}
Let $f \in R_d$ be a $k-$concise form with $2k<d$ and let $I=\Ann(f) \subset Q$. Let $J =(I_{d-k})\subset Q$ be the ideal generated by the degree $d-k$ part of $I$. If $J_{l}^{sat}\neq \emptyset$ for some $l \leq k$, then \[\text{cr}(f)> a_k=\binom{n+k}{k}.\]
\end{lema}
\begin{proof}

 Let $I=\Ann_f$  and consider the algebra $A = Q/I$. Let $a_i = \dim A_i$  Since $A$ is Gorenstein, we get $a_k=a_{d-k}$, by Poincaré duality. Let $B = Q/J$ and $b_i=\dim B_i$, we get that $b_k=\binom{n+k}{k}$ and $b_{d-k}=a_{d-k}$.

Let $K \subset I = \Ann_f$ be any saturated ideal satisfying the definition of cactus rank for $f$, that is, the zero dimensional scheme $X$ defined by $K$ has length $cr(f)$ and $f \in <X>$.  We know that the Hilbert function of $Q/K$ is non decreasing and stabilizes in the constant polynomial $\ell(K)=cr(f) \in \N$. 
Suppose that $\text{cr}(f)\leq a_k$. Thus, 
 \[\dim(Q/K)_{d-k} \leq cr(f) \leq  a_k=b_{d-k} = \dim (Q/J)_{d-k}.\]
 
 On the other hand, $K_{d-k} \subset I_{d-k}$, hence  \[\dim(Q/K)_{d-k} \geq \dim (Q/J)_{d-k}.\]
 Therefore we get \[\dim(Q/K)_{d-k} = \dim (Q/J)_{d-k}.\]
 Which gives us $K_{d-k}=J_{d-k}$, that is $J \subset K$, since $J$ is generated in degree $d-k$. Then, we get $J^{sat} \subset K^{sat} = K$, since $K$ is saturated. Since $f$ is $k-$ concise and $K\subset I$, we have
$$J_l=K_l=I_l=0.$$
For all $l \leq k$. It is a contradiction. Therefore, $cr(f)> a_k$.
\end{proof}

\begin{thm}\label{thm:main_A}
 Let $f \in R_d$ be a $k$-concise homogeneous form, with $2k \leq d$. 
 If $\hess_f=0$, then 
 \[ cr(f) > \binom{n+k}{k} .\]
 In particular, if $\underline{rk}(f)\leq \binom{n+k}{k}$, then $f$ is wild. 
\end{thm}

\begin{proof}  Consider the algebra $A = Q/I$ and $B=Q/J$ and let $a_i = \dim A_i$ and $b_i=\dim B_i$. Since $A$ is Gorenstein we get $a_k=a_{d-k}$, by Poincaré duality. Since $I_k=J_k=0$, by hypothesis, we get $a_k=b_k$, and by construction, $a_{d-k}=b_{d-k}$. Therefore, $b_k=b_{d-k}$. \\
 
 Since $\hess_f=0$, by Lemma \ref{lema:hess}, we know that $J^{sat}$ contains a linear form. By Lemma \ref{lema:k_concise}, the result follows. 
 
 \end{proof}

The following Corollary is one the main results of \cite{HMV} (see \cite[Theorem 3.5]{HMV}).

\begin{cor}
 Let $f\in R_d$ be a concise form with minimal border rank. If $\hess_f=0$, then $f$ is wild.
\end{cor}

\begin{proof}
Minimal border rank means $\underline{rk}(f)=a_1$. Since $f$ is $1$-concise and $\hess_f=0$, by Theorem \ref{thm:main_A}, we get  $cr(f)>a_1$. 
 \end{proof}

In low degree it seems to be hard to construct examples of wild forms with vanishing Hessian whose border rank is not minimal. 
On the other hand, in high degree we get families of such forms. 

\begin{ex}\rm 
 Consider the forms $f_d \in \C[x,y,z,u,v]_{d^2-1}$ given by $f_d=(xu^d+yu^{d-1}v+zv^d)^{d-1}$ we checked with Macaulay2 for several values $d$ that $f_d$ is $(d-1)$-concise. If this is true in general, then, by Theorem \ref{thm:main_A}, $cr(f)> \binom{d+3}{4}$. The $f_d=g^{d-1}$ is a $(d-1)$-th power of a form $g=xu^d+yu^{d-1}v+zv^d$ that we know it has vanishing Hessian. Indeed, by Gordan-Noether criteria, since the partial derivatives of $g$ satisfy $g_x^{d-1}g_z=g_y^d$, they are algebraically dependent, therefore, $\hess f_d = 0$. Moreover, we choose $g_d=xu^d+yu^{d-1}v+zv^d$ since its polar image has degree $d$, if the polar degree was lower, then the $f_d$ could not be $(d-1)$-concise.  On the other hand, by Proposition \ref{prop:border}, we get $\underline{rk}(f)\leq (d-1)(d^2+1)$. For any $d\geq 17 $ we get $\underline{rk}(f)\leq cr(f)$. For $d=17$ we checked the $16$-conciseness of $f_{17}$ which implies that $f_{17}$ is wild with border rank non minimal.  In this case $cr(f_{17})>a_{16}=4845$ and $\underline{rk}(f)\leq 4640$, hence $f$ is wild.
\end{ex}

The next example is related to Gordan-Noether original approach (see \cite{GN} and \cite[\S 2.3]{CRS}).  

\begin{defin} \label{def:GN} Let $R = \C[x_0,\ldots,x_t,u,v]$ with natural bi-grading. 
Let $Q_l = x_0M_{l0}+\ldots+x_tM_{lt}\in R_{(1,e-1)}$ with $l=1,\ldots,t-m$ be generic forms given by Gordan-Noether machinery (see \cite[\S 2.3]{CRS}). 
Let $d=\mu e$ and let $P_{\mu}(z_1,\ldots,z_s)$  be a generic form of degree $\mu$. 
A generic GN hypersurface of type $(t+2,t,m,e)$ and degree $d$ is defined by: 
 \[f=P_{\mu}(Q_1,\ldots,Q_{t-m}).\]
\end{defin}

\begin{ex}\rm 
 Consider a generic GN polynomial of type $(t+2,t,t-2,e)$, and degree $d=4e$, it means that there are two Perazzo polynomials with vanishing Hessian, $Q_1,Q_2\in \C[x_0,x_1,\ldots,x_t,u,v]_{(1,e)}$ given by Gordan-Noether machinery and a generic quartic polynomial $P(z_1,z_2)$ such that $f=P(Q_1,Q_2)$. By the genericity of $Q_1,Q_2$ and $P$, $f$ is $2$-concise. By \cite[Proposition 2.9]{CRS}, $\hess_f=0$. For $s=28$ and $e=30$, we get: \[cr(f)=496 > 488=\underline{rk}(f).\]
 \end{ex}

  Let $P(z_1,z_2)$ be a generic quartic polynomial, let $Q_i\in \C[x_0,x_1,\ldots,x_s,u,v]_{(1,e)}$ be generic Perazzo polynomials given by Gordan-Noether machinery, with $e=2 \lfloor \frac{s}{2} \rfloor$ and let $f=P(Q_1,Q_2)$ be a generic GN polynomial of type $(t+2,t,t-2,e)$ and degree $d=4e$.
  
 \begin{cor} With the previous notation, let $f=P(Q_1,Q_2)$ be a degree $4e$ generic GN polynomial. If $s\geq 28$, then $f$ is wild. 
  \end{cor}

  \begin{proof} The genericity of $Q_1,Q_2$ and $P$ implies that $f$ is $2$-concise. In fact, by Sylvester Theorem, \ref{thm:sylvester}, $P=l_1^4+l_2^4$ and we write $Q_1=x_0M_0+\ldots+x_tM_t$ and $Q_2=x_0N_0+\ldots+x_tN_t$, to simplify the notation. We get \[X_iX_j(f)=12(M_iM_jQ_1^2+N_iN_jQ_2^2.)\]
  Suppose that $\sum c_{ij}X_iX_j(f)=0$, then, using the bi-grading we get $\sum c_{ij}M_iM_j=0$ and $\sum c_{ij}N_iN_j=0$, which implies $c_{ij}=0$. \\
   By \cite[Proposition 2.9]{CRS}, $\hess_f=0$. From Proposition \ref{prop:border}, \[\underline{rk}(f) \leq 4(4(e+1)+2)=16e+40.\]
By Theorem \ref{thm:main_A}, \[cr(f)>\binom{s+4}{2}.\] 
For $s \geq 28$, 
\[cr(f) > \underline{rk}(f).\]
  \end{proof}

\subsection{$k$-concise wild forms with degenerated mixed Hessian} In this section we construct wild forms with non vanishing Hessian.

\begin{lema}\label{lema:k_concise2}
Let $f\in R_d$ be a $k-$concise form with $2k<d$. Let $I=\Ann(f)\subset Q$ and $A=Q/I$. Suppose that $\Hilb(A)$ is unimodal. Let $J =(I_{\leq d-k})\subset Q$ be the ideal generated by the graded parts of degree  $\leq  d-k$ of $I$. If $J_{l}^{sat}\neq \emptyset$ for some $l \leq k$, then \[\text{cr}(f)> a_k=\binom{n+k}{k}.\]
\end{lema}

\begin{proof}

 Let $I=\Ann_f$  and consider the algebra $A = Q/I$ and denote $a_i = \dim A_i$.  Since $A$ is Gorenstein, we get $a_k=a_{d-k}=\binom{n+k}{k}$, by Poincaré duality and by the $k$-conciseness of $f$. Let $B = Q/J$ and $b_i=\dim B_i$, we get that $b_k=\binom{n+k}{k}$, since $I_K=J_k=0$ by the  $k$-conciseness of  $f$. For $s \in\{k+1, \ldots,d-k\}$ we get $b_{s}=a_{s}$, notice also that $a_k \leq a_s$, since $\Hilb(A)$ is unimodal.

Let $K \subset I = \Ann_f$ be any saturated ideal satisfying the definition of cactus rank for $f$, that is, the zero dimensional scheme $X$ defined by $K$ has length $cr(f)$ and $f \in <X>$.  We know that the Hilbert function of $Q/K$ is non decreasing and stabilizes in the constant polynomial $\ell(K)=cr(f) \in \N$. 
Suppose that $\text{cr}(f)\leq a_k$. For any $s \in\{k+1, \ldots,d-k\}$, we get
 \[\dim(Q/K)_{s} \leq cr(f) \leq  a_k\leq a_s = \dim (Q/J)_{s}.\]
 
 On the other hand, $K_{s} \subset I_{s}$, hence  \[\dim(Q/K)_{s} \geq \dim (Q/J)_{s}.\]
 Therefore we get \[\dim(Q/K)_{s} = \dim (Q/J)_{s}.\]
 Which gives us $K_{s}=J_{s}$, that is $J \subset K$, since $J$ is generated in degree $\{k+1, \ldots,d-k\}$. Then, we get $J^{sat} \subset K^{sat} = K$, since $K$ is saturated. Since $f$ is $k-$concise and $K\subset I$, we have
$$J_l=K_l=I_l=0.$$
For all $l \leq k$. It is a contradiction. Therefore $cr(f)> a_k$.
\end{proof}

\begin{thm}\label{thm:main_C}
 Let $f \in R_d$ be a $k$-concise homogeneous form with $2k \leq d$ and let $l,s$ be integers such that $l\leq k\leq s$ and $s+l\leq d$. 
 Let $I=\Ann(f)$ and $A=Q/I$ and suppose that $\Hilb(A)$ is unimodal. 
 Suppose that $\Hess_f^{(l,s)}$ is degenerated, or equivalently, for a generic $L \in A_1$, the map $\bullet L^{}:A_l \to A_{d-s}$ is not injective. Then: 
 \[ cr(f) > \binom{n+k}{k}.\]
 In particular, if $\underline{rk}(f)\leq a_k$, then $f$ is wild. 
\end{thm}

\begin{proof} Let $I=\Ann_f$  and consider the algebra $A = Q/I$. Let $a_i = \dim A_i$  Since $A$ is Gorenstein we get $a_k=a_{d-k}=\binom{n+k}{k}$, by Poincaré duality. Let $J=(I_{\leq d-k})$ be the ideal generated by the pieces of $I$ in degree $\leq d-k$. Let $B = Q/J$ and $b_i=\dim B_i$, we get that $b_k=\binom{n+k}{k}$ and $b_{d-k}=a_{d-k}$. By hypothesis we have 
\[a_l=b_l \leq a_k=b_k\leq a_s=b_s=a_{d-s}=b_{d-s}.\]

By Lemma \ref{lema:hess}, there is $\gamma \in I_l^{sat}$. By hypothesis $s\geq k$, therefore, $d-s \leq d-k$, which implies $I_{d-s}=J_{d-s}$, hence $\gamma \in J^{sat}_{l}$. The result follows from Lemma \ref{lema:k_concise2}.
\end{proof}

The first example of a form with vanishing second Hessian whose Hessian is non vanishing was given by Ikeda in \cite{Ik}, see also \cite{MW,Go} 
for further discussions. 

\begin{ex}\rm 

Let $f=xu^3v+yuv^3+x^2y^3\in \C[x,y,u,v]_5$. Let $A=Q/\Ann_f$, we get \[\Hilb(A) = (1,4,10,10,4,1).\]
Therefore $f$ is $2$-concise. We know that $\hess_f^2=0$. By Proposition \ref{prop:border}, $\underline{rk}(f) \leq 7$. By Theorem \ref{thm:border_monomials}, $\underline{rk}(x^2y^3) =3$, then $\underline{rk}(f)\leq 10$. By Theorem \ref{thm:main_C} we get that $cr(f)>10$, therefore $f$ is wild.

\end{ex}

In \cite[Theorem 2.3]{Go}, the first author generalized the Ikeda's example, introducing a series of forms with vanishing Hessian of order $k$. They are called exceptional polynomials of order $k$ and degree $d$. 

\[f=\displaystyle \sum_{i=1}^m x_iM_i + h(x). \]

If we choose $h$ wisely, then we get $2$-concise exceptional polynomials. It is easy to control the border rank of such polynomials and obtain new examples of wild forms without vanishing hessian.  

\begin{ex}\rm 
 Let $f= xu^{5}v+yu^{3}v^3+zuv^5+\displaystyle \sum_{i=1}^6 l_i^7 \in \C[x,y,z,u,v]_7$ with $l_i \in \C[x,y,z]$ generic linear forms. 
 We checked, using Macaulay2, that $f$ is $2$-concise and that the Hilbert vector of the algebra is unimodal. By Theorem \cite[Theorem 2.3]{Go}, $\hess^2_f =0 $, which can also be checked directly. By Proposition \ref{prop:border}, 
 \[\underline{rk}(xu^{5}v+yu^{3}v^3+zuv^5)\leq 9. \]
 Hence, $\underline{rk}(f)\leq 15$. By Theorem \ref{thm:main_C}, $cr(f)>15$. Therefore, 
 $f$ is wild. 
\end{ex}

Generalizing this idea we get the following:

\begin{cor}\label{cor1}
 Let $f \in \C[x_1,\ldots,x_{n},u,v]_{d+2}$ be a exceptional form of degree $d+2$ with $d=2n-1>3$ given by: 
 \[ f=x_1u^dv+x_2u^{d-2}v^3+\ldots+x_nuv^{d}+h. \]
 With $h=\displaystyle \sum_{i=1}^{\binom{n+1}{2}}l_i^{d+2} \in\C[x_1,\ldots,x_n]$ where $l_i$ are generic linear forms.
 Then $f$ is wild. 
\end{cor}

\begin{proof}

For such exceptional form, it is easy to see that if $h \in \C[x_1,\ldots,x_n]_{d+2}$ is $2$-concise, then $f$ is $2$-concise.
The Hilbert vector of the associated AG algebra is unimodal (see \cite{Go}). 
Since $h=\displaystyle \sum_{i=1}^{\binom{n+1}{2}}l_i^{d+2}$ and $l_1\in \C[x_1,\ldots,x_n]_1 $ are generic, then it is $2$-concise. By \cite[Theorem 2.3]{Go}, $\hess_f^2=0$. 
By Proposition \ref{prop:border}, we get 
\[\underline{rk}(f)\leq (d+2)+2+\underline{rk}(h)\leq 2n+3 + \binom{n+1}{2} = \binom{n+3}{2}.\]
 Since $a_2=\binom{n+3}{2}$, by Theorem \ref{thm:main_C}, $cr(f) >\binom{n+3}{2}$. The result follows. 
 
\end{proof}

Also in \cite{Go}, the author generalized for higher Hessians some classical constructions of forms with vanishing Hessians tracing back to Gordan-Noether and Perazzo's counter examples to Hesse's claim. They are called GNP polynomials. 

\begin{prop}\cite[Prop. 2.5]{Go}\label{prop:GNP}
Let $f \in \C[x_0,\ldots,x_n,u_1,\ldots,u_m]_{k,e}$ a bi-graded form of bi-degree $(k,e)$ with $k<e$.
Let $f=\displaystyle \sum_{i=1}^s f_ig_i$ with $f_i \in \C[x]$ and $g_i \in \C[u]$, if $s> \binom{m+k-1}{k}$, then $\hess_f^k=0$.  
\end{prop}

\begin{ex}\rm 
Consider $M_i\in \C[x,y,z]_4$ with $i=0,\ldots,14$, be all the quartic monomials in $3$ variables and let 
\[f= \displaystyle \sum_{i=0}^{14} M_iu^{14-i}v^i \in \C[x,y,z,u,v]_{18}.\]
 We checked, using Macaulay2, that $f$ is $4$-concise.  By Prop \ref{prop:GNP}, $\hess^4_f=0$. By Theorem \ref{thm:main_C}, $cr(f)> \binom{4+4}{4}=140$. 
 By Proposition \ref{prop:border}, $\underline{rk}(f)\leq 4.(18+2)=80$. We get that $f$ is wild. 
\end{ex}

\begin{cor}\label{cor2}
 Let $M_i \in \C[x_0,\ldots,x_n]_k$ with $i=0,\ldots,b-1$ be all the monomials of degree $k$, where $b=\binom{n+k}{k}$. Let 
 \[f= \displaystyle \sum_{i=0}^{b-1} M_iu^{b-i}v^i \in \C[x,y,z,u,v]_{b-1+k}.\]
 If $\binom{n+k+2}{k}>k[(k+1)+\binom{n+k}{k}]$, then $f$ is wild. 
\end{cor}

\begin{proof}
 We want to show that $f$ is $k$-concise, that is, $a_k=\binom{n+k+2}{k}$. Consider the decomposition of $A_k$ given by the bi-grading of $f$:
 \[A_k=A_{(k,0)}\oplus\ldots\oplus A_{(i,k-i)}\oplus \ldots \oplus A_{(0,k)}.\]
 By the choice of all the monomials in both variables, we get that \[\dim A_{(i,k-i)} = \dim A_{(0,k-i)}\dim A_{(i,0)}=(k-i+1)\binom{n+i}{i}.\] Therefore \[\dim A_k = \displaystyle \sum_{i=0}^k (k-i+1)\binom{n+i}{i} =\binom{n+k+2}{k}.\] 
 By Proposition \ref{prop:GNP},  $\hess_f^k=0$. By Proposition \ref{prop:border}, $\underline{rk}(f)\leq k[k+b-1+2] = k[(k+1)+\binom{n+k}{k}]$. The result follows from Theorem \ref{thm:main_C}. 
\end{proof}

{\bf Acknowledgments}.
We wish to thank F. Russo for his insightful suggestions and conversations on the subject.


\begin{thebibliography}{HMMNWW}



\bibitem[AH1]{AH1} J. Alexander, and  A. Hirschowitz  {\it La methode d'Horace eclatee: application a l'interpolation en degre quatre. Inventiones mathematicae, 107(1)}, 1992, 585-602.

\bibitem[AH2]{AH2}J. Alexander, and  A. Hirschowitz  {\it Un lemme d'Horace différentiel: application aux singularités hyperquartiques de $\mathbb{P}^5$.} J. Algebraic Geom, 1(3), 1992, 411-426.

\bibitem[AH3]{AH3}J. Alexander, and  A. Hirschowitz, Polynomial interpolation in several variables. Journal of Algebraic Geometry, 4(2), 1995, 201-222.


\bibitem[BB]{BB} W. Buczyńska, J. Buczyński, On differences between the border rank and the smoothable rank of a polynomial. Glasgow Mathematical Journal, 57(2), 2015,  401-413.

\bibitem[Be]{Be} N. Bergeron (2003).{\it Lefschetz properties for arithmetic real and complex hyperbolic manifolds}. International Mathematics Research Notices, 2003(20), 1089-1122.


\bibitem[BL]{BL} M. Boij and D. Laksov, {\it Nonunimodality of graded Gorenstein Artin algebras}, Proc. Amer. Math. Soc. 120 (1994), 1083--1092


\bibitem[BR]{BR} A. Bernardi, K. Ranestad, (2013). {\it On the cactus rank of cubic forms. Journal of Symbolic Computation}, 50, 291-297.


\bibitem[BBM]{BBM} A. Bernardi, J. Brachat, B. Mourrain (2014). {\it A comparison of different notions of ranks of symmetric tensors}. Linear Algebra and its Applications, 460, 205-230.

\bibitem[BCMT]{BCMT} J. Brachat, P. Comon, B. Mourrain and E. Tsigaridas, {\it Symmetric tensor decomposition.} Linear Algebra and its Applications, 433(11-12), (2010) 1851-1872.





\bibitem[BBT]{BBT} W. Buczyńska, J. Buczyński and  Z. Teitler {\it Waring decompositions of monomials.} Journal of Algebra, 378, 45-57.



\bibitem[BN]{BN} E. Babson, E. Nevo, E. (2010). {\it Lefschetz properties and basic constructions on simplicial spheres}. Journal of Algebraic Combinatorics, 31(1), 111.



\bibitem[CRS]{CRS} C. Ciliberto, F. Russo, A. Simis,
{\it  Homaloidal hypersurfaces and hypersurfaces with vanishing Hessian}, Adv. in Math.
2 18 (2008), 1759-1805.


\bibitem[CGLM]{CGLM} P. Comon, G. Golub, L. H. Lim, and B. Mourrain,  {\it Symmetric tensors and symmetric tensor rank.} SIAM Journal on Matrix Analysis and Applications, (2008), 30(3), 1254-1279.


 \bibitem[CS]{CS} G. Comas, , M. Seiguer,   {\it On the rank of a binary form. Foundations of Computational Mathematics,} 11(1), (2011), 65-78.

 
\bibitem[DGI]{DGI} A. Dimca, R. Gondim, G. Ilardi {\it Higher order Jacobians, Hessians and Milnor algebras}, Collect. Math. (2019). https://doi.org/10.1007/s13348-019-00266-1
 

\bibitem[Do]{Do} Dolgachev, I. V. (2004). {\it Dual homogeneous forms and varieties of power sums}. Milan Journal of Mathematics, 72(1), 163-187.



\bibitem[ECG]{ECG} E. Carlini,  M. V. Catalisano, and A. V. Geramita. {\it The solution to the Waring problem for monomials and the sum of coprime monomials.} Journal of algebra 370, 2012: 5-14.


\bibitem[Go]{Go} R. Gondim {\it On higher Hessians and the Lefschetz properties},  Journal of Algebra 489 (2017), 241–263. 

\bibitem[GZ]{GZ} R. Gondim, G. Zappalà {\it Lefschetz properties for Artinian Gorenstein algebras presented by quadrics} Proc. Amer. Math. Soc. 146 (2018), no. 3, 993–1003. 

\bibitem[GZ2]{GZ2} R. Gondim, G. Zappalà {\it On mixed Hessians and the Lefschetz properties}, arXiv:1803.09664


\bibitem[GN]{GN} P. Gordan, M. N\" other,
{\it Ueber die algebraischen Formen, deren Hesse'sche Determinante
identisch verschwindet}, Math. Ann. 10 (1876), 547--568.

\bibitem[GR]{GR} A. Garbagnati, F. Repetto, {\it A geometrical approach to
Gordan--Noether's and Franchetta's contributions to a question posed by Hesse}, Collect. Math.
60 (2009), 27--41.
 
\bibitem[GRu]{GRu} R. Gondim, F. Russo, {\it Cubic hypersurfaces with vanishing
Hessian}, Journal of Pure and Applied Algebra  219 (2015), 779-806. 


\bibitem[He]{He} O. Hesse, {\it
Zur Theorie der ganzen homogenen Functionen}, J. reine angew.
Math. 56 (1859), 263--269.

\bibitem[HL]{HL} M. Harris, J. S. Li (1998). {\it A Lefschetz property for subvarieties of Shimura varieties}. Journal of algebraic geometry, 7(1), 77.



\bibitem[HMMNWW]{HMMNWW} T. Harima,  T. Maeno, H.  Morita, Y.  Numata, A.  Wachi, J.  Watanabe, {\it The Lefschetz properties}, 
Lecture Notes in Mathematics 2080.
Springer, Heidelberg, 2013,  xx+250 pp.

\bibitem[HMV]{HMV} H. Huang, M. Micha\l{}ek, E. Ventura, {\it Vanishing Hessian, wild forms and their border VSP} Vanishing Hessian, wild forms and their border VSP

\bibitem[Ik]{Ik} H. Ikeda, {\it Results on Dilworth and Rees numbers of artinian local rings},
Japan. J. Math. 22 (1996), 147 158.

\bibitem[IK]{IK} A. Iarrobino, V. Kanev, (1999). {\it Power sums, Gorenstein algebras, and determinantal loci} Springer Science and Business Media.



\bibitem[Ka]{Ka} H. Kasuya,  {\it Minimal models, formality, and hard Lefschetz properties of solvmanifolds with local systems}. Journal of Differential Geometry, (2013). 93(2), 269-297.

\bibitem[Ku]{Ku} J. P. Kung, {\it Canonical forms for binary forms of even degree. In Invariant theory}  Springer, Berlin, Heidelberg (1987).(pp. 52-61).

\bibitem[KN]{KN} M. Kubitzke, E. Nevo , {\it The Lefschetz property for barycentric subdivisions of shellable complexes}. Transactions of the American Mathematical Society, (2009). 361(11), 6151-6163.

\bibitem[La]{La} K. Lamotke {\it Topology of complex algebraic Varieties after S. Lefschetz}, Topology 20 (1981), 15--51.

\bibitem[LT]{LT} J. M. Landsberg, Z. Teitler, {\it On the ranks and border ranks of symmetric tensors}. Foundations of Computational Mathematics, (2010). 10(3), 339-366.

\bibitem[Lo]{Lo} C. Lossen, {\it When does the Hessian determinant vanish identically? (On Gordan and Noether's Proof of Hesse's Claim)}, Bull. Braz.
Math. Soc. 35 (2004), 71--82.

\bibitem[MMN]{MMN} J. Migliore, R. Miró-Roig, U.  and Nagel, {\it Monomial ideals, almost complete intersections and the weak Lefschetz property.} Transactions of the American Mathematical Society, 363(1), 2011 pp.229-257.



\bibitem[MW]{MW} T.  Maeno, J.  Watanabe, {\it Lefschetz  elements of  artinian  Gorenstein algebras and Hessians of homogeneous polynomials}, Illinois J. Math. 53 (2009), 593--603.	

\bibitem[NW]{NW} Numata, Y., ; Wachi, A. (2007). {\it The strong Lefschetz property of the coinvariant ring of the Coxeter group of type H4}. Journal of Algebra, 318(2), 1032-1038.

\bibitem[Pe]{Pe} U. Perazzo, {\it Sulle variet\` a cubiche la
cui hessiana svanisce identicamente}, Giornale di Matematiche (Battaglini)  38  (1900),
337--354.



\bibitem[Ru]{Ru} F. Russo {\it On the Geometry of Some Special Projective Varieties}, Lecture Notes of the Unione Matematica Italiana, vol. 18 Springer (2016).



\bibitem[Se]{Se}B. Segre, {\it Some Properties of Differentiable Varieties and Transformations}. Erg. Math.
Springer, Berlin, Heidelberg, 1957



\bibitem[St]{St} R.Stanley,  {\it Weyl groups, the hard Lefschetz theorem, and the Sperner property},  SIAM J. Algebraic Discrete Methods 1 (1980), 168--184.

\bibitem[St2]{St2} R. Stanley, {\it Hilbert functions of graded algebras}, Adv. in Math. 28 (1978), 57--83

\bibitem[Syl]{Syl} J. J. Sylvester, {\it On a remarkable discovery in the theory of canonical forms and of hyperdeterminants}, Philosophical Magazine, vol. I, 1851. Paper 41: 265-283.

\bibitem[Syl2]{Syl2} J. J. Sylvester, {\it Sur une extension d’un théoreme de Clebsch relatif aux courbes du quatrieme degré.} Comptes Rendus, Math. Acad. Sci. Paris 102 (1886): 1532-1534.



\bibitem[Wa1]{Wa1} J. Watanabe, {\it A remark on the Hessian of homogeneous polynomials}, in {\it The Curves
Seminar at Queen's},  Volume XIII, Queen's Papers in Pure and Appl. Math. 119 (2000), 171--178.


\end{thebibliography}
\end{document}